\documentclass{article}

\usepackage[english]{babel}
\usepackage[letterpaper,top=2cm,bottom=2cm,left=3cm,right=3cm,marginparwidth=1.75cm]{geometry}

\usepackage{amsmath}
\usepackage{amsthm,amssymb}
\usepackage{mathtools}
\usepackage{graphicx}
\usepackage[colorlinks=true, allcolors=blue]{hyperref}
\usepackage{makecell}
\usepackage{hhline}
\usepackage{comment}
\usepackage{listings}
\usepackage{bbold}
\usepackage{soul}

\theoremstyle{definition}
\newtheorem{definition}{Definition}[section]
\newtheorem{defi}[definition]{Definition}
\newtheorem{theorem}[definition]{Theorem}

\newtheorem{observation}[definition]{Observation}

\newtheorem{lemma}[definition]{Lemma}
\newtheorem{remark}[definition]{Remark}

\newtheorem{conj}[definition]{Conjecture}

\newcommand{\eps}{\varepsilon}
\newcommand{\bbP}{\mathbb P}

\DeclarePairedDelimiter\abs{\lvert}{\rvert}

\title{Bisection width, max-cut and internal partitions of 5-regular graphs}
\author{Endre Csóka\thanks{HUN-REN Alfréd Rényi Institute of Mathematics, Budapest, Hungary. },
Panna Tímea Fekete\thanks{HUN-REN Alfréd Rényi Institute of Mathematics, Budapest, Hungary. The author was supported by the ERC Synergy Grant No.\ 810115 -- DYNASNET. E-mail: {\tt fekete.panna.timea@renyi.hu}.},
 Zolt\'an L\'or\'ant Nagy\thanks{ELTE Linear Hypergraphs  Research Group,
		E\"otv\"os Lor\'and University, Budapest, Hungary. The author is supported by the University Excellence Fund of Eötvös Loránd University and the János Bolyai Research Scholarship of the Hungarian Academy of Sciences. 	E-mail: {\tt nagyzoli@cs.elte.hu}.},
Levente Szemerédi\thanks{HUN-REN Alfréd Rényi Institute of Mathematics, Budapest, Hungary; E\"otv\"os Lor\'and University, Budapest, Hungary. The author was supported by the EKÖP-24 university excellence scholarship program of the Ministry for Culture and Innovation from the source of the National Research, Development and Innovation Fund. E-mail: {\tt szelev@caesar.elte.hu}.}}
\date{}
\begin{document}
\maketitle

\begin{abstract}
In this paper, we 
present a new factor of IID process based on the local algorithm introduced by Díaz, Serna, and Wormald (2007). This new approach allows us to improve the previously known upper bounds on the minimum and maximum bisection width and the maximum cut of random $d$-regular graphs for $d>4$ by introducing a new recoloring phase after the termination of the original algorithm. 
As an application, we show that random 5-regular graphs asymptotically almost surely admit an internal partition, i.e., a partition of the vertex set into two nonempty classes so that every vertex has at least half of its neighbors in its own class. 
\end{abstract}

\section{Introduction}

For a given graph $G=(V,E)$, a {\em bisection} of $G$ is a partition of its vertex set $V$ into two disjoint parts $V_1$ and $V_2$ of equal or almost equal cardinality, $|V_1|=\left\lfloor \frac{|V|}{2}\right\rfloor$, $|V_2|=\left\lceil \frac{|V|}{2}\right\rceil$. The size of a bisection is the number of edges that cross between the two parts. A minimum bisection is a bisection of minimal size and its cardinality is called the {\em bisection width.}
It is well known that determining the bisection width is an NP-complete problem even for regular graphs \cite{johnson1979computers}.

The problem of determining the bisection width has a long history and is closely related to several important concepts and results in combinatorics, statistical physics, and computer science \cite{dembo2017extremal, diaz2007bounds, lyons2017factors, montanari2004instability, zdeborova2010conjecture}.
Following the work of Buser \cite{buser1978cubic}, Bollobás \cite{bollobas1988isoperimetric} gave bounds on the so-called isoperimetric number $$i(G):= \min_{U\subset V(G), |U|\leq |V|/2} |\partial{U}|/|U|$$  of $d$-regular graphs, where $\partial{U}$ denotes the edge set joining $U$ and $V\setminus U$. 
The isoperimetric number can be considered as a scaled version of the min-cut problem.\\ For the max-cut problem, the aim is to determine the vertex partition which provides the most crossing edges. This problem is also NP-complete \cite{johnson1979computers}.
On the other hand, determining the maximum edge-cut and determining the bisection width is conjectured to be closely related in a quantitative manner as well. 
\begin{conj}[Zdeborov\'a and Boettcher \cite{zdeborova2010conjecture}] 
In a random regular graph, the
size of the max-cut is asymptotically equal to the number of edges minus the bisection width, asymptotically almost surely. 
\end{conj}

 Bounds on the size of the max-cut and on the bisection width  in general and in random (regular) graphs were  proved via various techniques, and it received substantial interest, see e.g., \cite{alon2003maximum, balla2024maxcut, jin2025beyond, raty2023positive} and the references therein. One classical tool to tackle the problem, especially in the sparse regime,  is to apply \textit{local algorithms }from the constructive direction.

In such processes, we start from a graph with some independent and identically distributed (i.i.d.) random seeds assigned to each vertex. Then we perform a randomized algorithm by applying the same local rule at each vertex, that is, in parallel rounds, every vertex $w$ gets a new label based on the previous labels and the i.i.d.\ random seeds of the vertices having distance at most $r$ from $w$.
 
In this paper, we present a novel local algorithm which improves previous results on the size of the max-cut and on the bisection width of random $d$-regular graphs when $d>4$ is a constant.

Local algorithms and factor of i.i.d.\ processes turned out to be very useful in finding near-optimal solutions to several central problems in random graphs concerning graph parameters, such as the maximum independent set or maximum cut. However, there is evidence that this approach has limitations.  Chen, Gamarnik Panchenko, and Rahman \cite{chen2019suboptimality} showed that for random uniform hypergraphs of average degree large enough, local algorithms cannot find almost maximal cuts.
Recently, Basso et al. applied Quantum Approximate Optimization Algorithm (QAOA) to the Max-Cut problem on large-girth $d$-regular graphs \cite{basso2022quantum}.
El Alaoui, Montanari, and Sellke solved the problem of computing a "near" maximum cut or a "near" minimum bisection of $d$-regular graphs in polynomial time in \cite{el2023local}, where the error term depends on the degree $d$, and the number of vertices of the graph $n$, thus it can be considered as a completion of our work for $d\gg 1$.

Interestingly, while for random $n$-vertex $d$-regular graphs, the maximum cut value converges in probability as $n \to \infty$, the analogous statement is only conjectured for the bisection width when $d>2$ holds \cite{bayati2010combinatorial}. 
For numerous applications concerning the problems above in the fields of computer science, statistical physics, and bioinformatics, we refer to \cite{ coja2022ising, klamt2004minimal, poljak1995maximum, sharan2003click}. 

In \cite{wormald1999models}, Wormald described a technique developed for translating many results for random $d$-regular graphs to $d$-regular graphs
with sufficiently large girth via analyzing the performance of so-called locally greedy algorithms for random regular graphs. 
Concerning the bisection width and the maximum cut, the algorithms and their analysis are due to Díaz, Serna, Do, and Wormald, and are detailed in \cite{diaz2003bounds} and \cite{diaz2007bounds}. For more recent algorithm variants, we refer to \cite{diaz2024minimum}.
Their main tool during the analysis is the celebrated \textit{differential equation method}. The approach was extended to several other graph parameters due to Hoppen and Wormald \cite{hoppen2018local}.

We note that such problems can also be approached in an alternative manner.
The asymptotic minimum bisection width of random $d$-regular graphs on $n \to \infty$ vertices which can be constructed by local algorithms, is equal to the minimum bisection width of the corresponding so-called Bernoulli-graphing.
Graphings are fundamental objects in sparse graph limit theory, see \cite{hatami2014limits}. 
This theory makes it obvious that these algorithms provide the same asymptotic bisection width for all graph sequences with large essential girth, because they all converge to the same local limit.
The same applies to the maximum cut, and several other graph parameters as well.
The main idea of our proof was inspired by this view; however, we decided not to follow this direction and use the limit object, as it would require to build a technical framework. 

The purpose of this paper is threefold. First, we revisit the local greedy algorithm of Díaz, Serna, and Wormald to add \textit{a terminal phase}, which enables us to improve their (lower) bounds, which yields the best lower bounds on the bisection width in random $d$-regular graphs up to our knowledge for every $d>4$. 

Secondly, we describe a \textit{new variant} of the aforementioned \textit{random greedy algorithm} in which linearly many vertices are allocated in two partition classes in each step, and hence the procedure terminates in a finite number of steps.
This approach has its roots in statistical physics and in graph limit theory, and it may be of independent interest. It also has some similarity to that of Lyons \cite{lyons2017factors}, and of Lichev and Mitsche, who hold the record for the case of cubic graphs \cite{lichev2023minimum}.
We formulate this algorithm for the  $d$-regular case for arbitrary integer values $d>2$, and present quantitative results for $d=5$. 

We also introduce a \textit{new technique to analyze} the performance of more cumbersome local algorithms, where the local rule may depend not only on the strict neighborhoods of the vertices.
 More importantly, the analysis of the coloring algorithm is obtained on the infinite $d$-regular tree, and we argue that by calculating the corresponding distribution of types over time over this object, the error term we made is negligible with high probability compared to the distribution over random $d$-regular graphs. 
 
Together with the terminal phase, we believe that the performance analysis would be hard to execute via the (standard version of the) differential equation method. 

 	As an application, we derive a result on the size of the maximum cut and on internal partitions of regular graphs. An {\em internal} or {\em friendly} partition of a graph is a partition of the vertices into two nonempty sets so that every vertex has at least as many neighbors in its own class as in the other one. It is known due to Linial and Louis \cite{linial2020asymptotically} that almost all $d$-regular graphs admit an internal partition if $d$ is even, and a similar statement holds for $d=3$. However, the problem with larger odd values of $d$ remained open. We derive the analogue result for $d=5$.

These ideas lead to the following results.

\begin{theorem}\label{bisect}

The bisection width of a random 5-regular graph on $n$ vertices is asymptotically almost surely (a.a.s.) smaller than $0.496488n$.

\end{theorem}

\begin{theorem}\label{maxcut}
 The size of the maximum cut of a random 5-regular graph on $n$
vertices is a.a.s. larger than $2.003n$.
\end{theorem}

As a consequence of Theorem \ref{bisect}, we show the following theorem.
\begin{theorem}\label{internal}
 Asymptotically almost every $5$-regular graph has an internal partition.
\end{theorem}

\bigskip

The organisation of the paper is as follows. In Section 2 we first describe the greedy coloring algorithm of Díaz, Serna, and Wormald and present a new recoloring phase which enables us to improve the bound on the bisection width. Then we introduce the new local algorithm, which, instead of coloring vertices in pairs, proceeds by coloring in a constant number of rounds, independently of the size of the vertex set. Then we prove Theorem \ref{bisect} by the analysis of the algorithm.

 Finally, in Section 3, we derive Theorem \ref{maxcut} and \ref{internal} as consequences of Theorem \ref{bisect}, and in Section 4, we discuss possible further improvements.

\section{A refined version of the Díaz-Serna-Wormald algorithm}

 In this section, we describe a coloring algorithm that gives us a small cut with (almost) equal sizes for regular graphs. First, we revisit a greedy coloring introduced by Díaz et al. \cite{diaz2007bounds} and describe a new phase for the algorithm. In this phase, we recolor some vertices, whose cardinality is proportional to $n$, in order to obtain a bisection of smaller size. 

Let us introduce some notation. We denote by $N(v)$ the neighbors of the vertex $v$, $N(v):=\{w\in V(G): vw\in E(G)\}$, while the closed neighborhood of $v$ is defined by $N[v]:=N(v)\cup\{v\}$. We define the $k$-neighborhood $N^k[v]$ as the set of vertices whose graph distance from $v$ is at most $k$, including the vertex $v$. The neighbors and $k$-neighborhood of a set of vertices is the union of the neighbors and $k$-neighborhoods of the individual vertices, respectively. $\Delta(G)$ denotes the maximal degree of the graph $G$.

 A \textit{bihole} in a bipartite graph $G=(A\cup B, E)$ is a balanced independent set $I$, that is, $|I\cap A|=|I\cap B|$.

\subsection{Greedy coloring}\label{subsec:greedy}

 Díaz et al. \cite{diaz2007bounds} analysed a randomized greedy algorithm which proceeds as follows.\\

 \noindent \textbf{Díaz-Serna-Wormald algorithm to obtain small bisection width.}
Take a $d$-regular graph $G$.
The vertices of $G$ are put into two bipartition classes, denoted by red and blue, in pairs.
Given a partial coloring of $G$ with colors red and blue, we classify its uncolored vertices according to the number of their red and blue neighbors. A vertex is \textit{of type $(r,b)$} if it has $r$ red and $b$ blue neighbors. We call two uncolored vertices a \emph{symmetric pair} if one of them is of type $(x,y)$ and the other is of type $(y,x)$.

At each step, we proceed by choosing a symmetric pair with types $(r,b)$ and  $(b,r)$ and color red the one with more red neighbors, the other blue. If they have the same number of red and blue neighbors, then one is colored red, the other blue, uniformly at random.
The values of $r$ and $b$ in the chosen type $(r,b)$ are determined according to a \textit{priority ordering for types}, and vertices are chosen uniformly at random from those which have the highest priority.

The priority order of types is defined as follows:
\begin{itemize}
\item type $(r,b)$ and type $(b,r)$ has the same priority,
\item if $r_1\leq b_1$ and $r_2\leq b_2$ then the priority of type $(r_1,b_1)$ is less than the priority of type $(r_2,b_2)$ if and only if either $b_1-r_1<b_2-r_2$, or both $b_1-r_1=b_2-r_2$ and $r_1<r_2$.
\end{itemize}

If there are no more symmetric pairs left, color half of the remaining vertices red and the other half blue uniformly at random. 

\subsection{A novel (terminal) phase: swapping wrongly placed vertices}\label{sec:novelphase}

 We introduce the notion of \textit{terminal types} for the vertices. A vertex is of {\em terminal type} $(r',b')$ if it has $r'$ red and $b'$ blue neighbors after the greedy coloring procedure terminates.
 Note that in the algorithm above, only uncolored vertices were clustered into types, while the terminal type is introduced for the colored vertices at the end of the algorithm, hence $r'+b'=d$ must hold.
\begin{defi}[Miscolored vertices]	Let us call those vertices {\em miscolored red $(r,b)$ vertices}, which were colored red while being of type $(r,b)$ with $r> b$, but have terminal type $(r',b')$ with $r'<b'$,  and likewise call {\em miscolored blue $(r,b)$ vertices} those which were colored blue while being of type $(r,b)$ with $b> r$, but have terminal type $(r', b')$ with $b'<r'$. 
\end{defi}

 In order to get a bisection with a smaller size compared to the one obtained from the greedy coloring, the idea is to recolor a subset of miscolored red and blue vertices to the opposite colors.

\begin{lemma}\label{indep} Suppose that $I$ is an independent set in the bipartite graph induced by partition classes, defined by the miscolored red and blue vertices. Iteratively replacing each vertex of the independent set to the opposite class decreases the number of edges in the cut defined by the red-blue coloring, hence the size of the bisection decreases by at least $|I|$ at the end of the swapping procedure.
\end{lemma}

\begin{proof} Since for every vertex $v\in I$, the color of its neighbors $N(v)$ stay intact during the replacement. Thus, the replacement  of every vertex $v\in I$ contributes independently by at least $1$ to the decrement of the size of the cut.
\end{proof}

\begin{theorem}[Axenovich et al., 2021. \cite{axenovich2021bipartite}] \label{biholeswap}
 For any fixed positive integer $\Delta$, there exists an integer $k_0(\Delta)$ such that the following holds.
Suppose that $H$ is a balanced bipartite graph on $k+k$ vertices and has maximal degree $\Delta$ where $k\ge k_0(\Delta)$. Then $H$ contains a bihole of size $\mu(\Delta)k+\mu(\Delta)k$,
where $\mu(3)>0.34116,$ $ \mu(4)>0.24716$, and in general, $\mu(\Delta)\ge c\frac{\log{\Delta}}{{\Delta}}$ for some positive constant $c>0$.
\end{theorem}

\subsection{The new algorithm}\label{subsec:newalg}

While the local algorithm of Díaz, Serna, and Wormald proceeded by coloring the vertices pair after pair, we apply a coloring procedure where we aim to color a subset of uncolored vertices of size approximately proportional to the whole vertex set, but keeping the priority-based rules as follows.

For the first (greedy) phase, take a positive, small real number $\eps>0,$ which we think of as the expected step size. 
 The number of steps is denoted by by $N\approx\eps^{-1}$, which is specified in Subsection \ref{transition}.
For each step $t$, let us introduce its \emph{dominant type} $D_t=\{r^d_t,b^d_t\}$, which is an unordered type of uncolored vertices.
At this point, we should only think of the list of dominant types $D_1,D_2,\ldots,D_{N}$ as parameters of the algorithm, we discuss how to determine them in Subsection \ref{transition}.
When we need the additional information whether a vertex of the dominant type has more red or blue neighbors, we denote the ordered dominant type $D_t^r$ and $D_t^b$ respectively. In the case when a vertex of the dominant type has the same red and blue neighbors, the two ordered dominant types are the same.
 For deciding which vertices amongst those of dominant type are getting colored we also need a sequence of threshold values $q_1,q_2,\ldots,q_N\in[0,1]$ which is, for now, also a parameter which we discuss more in Section \ref{transition}. For introducing local randomness, draw a sequence $s(v)=(s_1(v),s_2(v),\ldots,s_{N+1}(v))$ of real numbers from $(0,1)$ for every vertex $v\in G$ independently, uniformly at random.

In principle, the sequences $D_1,D_2,\ldots,D_N$ and $q_1,q_2,\ldots,q_N$ are determined in a way that at each step, approximately $\eps\cdot \abs{V(G)}$ uncolored vertices get colored and the dominant type $D_t$ is a type of high priority having enough vertices to be colored.

At each step $1\leq t\leq N$, we color some uncolored vertices according to the following rule. If the dominant type is symmetric, i.e.\ $D_t^r=D_t^b,$ then we apply
$$
\begin{cases}
v\to red  \text{ if } v \text{ is of type } D_t \text{ and } s_t(v)\in [0,q_t]\\
v\to blue \text{ if } v \text{ is of type } D_t \text{ and } s_t(v)\in (q_t,2q_t].
\end{cases}
$$
Otherwise we apply
$$
\begin{cases}
v\to red  \text{ if } v \text{ is of type } D_t^r \text{ and } s_j(v)\in [0,q_t]\\
v\to blue \text{ if } v \text{ is of type } D_t^b \text{ and } s_j(v)\in [0,q_t].
\end{cases}
$$

After the $N$ coloring steps, if a vertex $v$ is still uncolored, then we color it red if $s_{N+1}(v)\in[0,1/2]$ and color it blue otherwise.

Finally, after having colored all the vertices, we do a recoloring for a subset of miscolored vertices, as described in Subsection \ref{sec:novelphase}. We take the vertex set of the largest induced balanced empty subgraph of the bipartite graph, which is induced by the set of red vertices which have more than $d/2$ blue neighbors, and the set of blue vertices which have more than $d/2$ red neighbors. Then we swap the colors of these vertices.

\subsection{Modelling the algorithm on the infinite regular tree}

In order to prove that our algorithm gives the desired output, we notice first that for constant $d$, a random $d$-regular graph have high girth with high probability. This enables us to model the algorithm on the $d$-regular infinite tree and later analyze the order of the error we made by switching the underlying structure.

In this section, we do calculations using the different colored rooted infinite $d$-regular tree $G=(V,E).$ We call its root $v_0,$ and the neighbors of $v_0$ by $v_1,\dots,v_d.$

Let us denote the set of colored infinite $d$-regular trees $$\mathcal T^c=\{V\to\{\text{"U","R","B"}\}\}$$ where "U", "R" and "B" correspond to the colors "uncolored", red and blue, respectively.

Similarly, we write $$\mathcal T^{pc}=\{V\to\{\text{"U","R","B","unspecified"}\} \mid \text{all but finitely many "unspecified"}\}$$ for the set of partially colored $d$-regular trees for which we only know the color of finitely many vertices. Here we need to emphasize that we regard "uncolored" to be a color beside red and blue. Naturally, an element of $\mathcal T^{pc}$ can be identified as a subset of $\mathcal T^c$ by letting the "unspecified" vertices being any of "U", "R" or "B". With this convention, we make the set $\mathcal T^c$ measurable by taking the $\sigma$-algebra generated by $\mathcal T^{pc}$ to be the set of measurable subsets.

Let $C_t:\Omega\to\mathcal T^c$ be the random variable describing the coloring of the infinite tree after the $t^{th}$ step of the first phase of the new algorithm. Let us denote the law of $C_t$ by $\mu_t$. The measure $\mu_t$ is invariant under the symmetries of the infinite $d$-regular tree, in particular, changing the root and reordering its neighbors. Moreover, the role of red and blue is also symmetric thus $\mu_t$ is also invariant under changing all the colors "R" to "B" and "B" to "R".

We are interested in the coloring of some subset $S\subset V$ of the vertices. This is a marginal of the random variable $C_t$ by forgetting the colors of the vertices of $V\setminus S$. We denote this marginal variable $C_{t}(S)$ and its law $\mu_{t}(S)$. As a special case, when we only describe the color of one vertex we omit the set notation and just use the vertex. The next lemma states that we do not need to specify how $S$ is embedded in the (rooted) infinite $d$-regular tree, but only need its isomorphism type.

\begin{lemma}\label{lemma:marginal} Let $G'=(V',E')$ be a connected finite graph and $\iota:G'\hookrightarrow G$ an embedding into the infinite $d$-regular tree. Then the measure $\mu_t(G')$ is well-defined and independent of $\iota.$
\end{lemma}

\begin{proof}
	Similarly to the definitions of $\mathcal T^c$ we can define $\mathcal T^c_{G'}$ to be the set of colorings of $G'$ and $\mathcal T^{pc}_{G'}$ the set of partial colorings. We make the set $\mathcal T^c_{G'}$ measurable by taking the $\sigma$-algebra generated by $\mathcal T^{pc}_{G'}.$

	The embedding $\iota$ induces the map $\iota_*:\mathcal T^{pc}_{G'}\hookrightarrow\mathcal T^{pc}$ by setting the color of the vertices in $V\setminus\iota(V')$ to be "unspecified". With an abuse of notation, we extend $\iota_*$ to the whole $\sigma$-algebra generated by the domain, which is, because of the finiteness of $G'$, the power set of $\mathcal T^c_{G'}.$ Any element $T$ of the power set can uniquely be written as a finite union of some basis elements $T_i\in\mathcal T^c_{G'}$: $T=\bigcup_{i}\{T_i\}$. Then we define $\iota_*(T)=\bigcup_i\{\iota_*(T_i)\}$. Now, $\iota_*$ clearly commutes with the union operation.

	For the consistency of the intersections, it is enough to check the pairwise intersections of basis elements. Let $T_1,T_2\in\mathcal{T}^c_{G'}$ be two colorings. We have to check if $\iota_*(\{T_1\}\cap\{T_2\})=\iota_*(\{T_1\})\cap\iota_*(\{T_2\})$ holds. When $T_1=T_2$ it is clearly true. If the two colorings are not the same, then the LHS is the empty set by definition. The RHS has the meaning that "the vertices $\iota(V')$ are colored according to $\iota_*(T_1)$ and are colored according to $\iota_*(T_2)$". Because $T_1\neq T_2$, this is also the empty set. We showed that $\iota_*$ commutes with the union and the intersection, thus all subsets of $\mathcal T^c_{G'}$ have a well-defined measurable image under $\iota_*$. This makes the pullback measure $\mu_t(G')=\mu_t\circ\iota_*$ well defined on $G'$.

 Now, let $\kappa:G'\hookrightarrow G$ be another embedding. There is a symmetry $\Sigma$ of the $d$-regular infinite tree $G$ for which $\kappa = \Sigma\circ\iota$ holds. Then $\mu_t\circ\kappa_*=\mu_t\circ{(\Sigma\circ\iota)_*}=\mu_t\circ(\Sigma_*\circ\iota_*)=(\mu_t\circ\Sigma_*)\circ\iota_*=\mu_t\circ\iota_*$ where $\Sigma_*$ is the map induced by $\Sigma$ on the generated $\sigma$-algebra of $\mathcal T^{pc}$ to itself. The second last equality makes sense because the inverse of $\Sigma$ is also measurable and the last one is true because $\mu_t$ is invariant under the symmetries of $G$. This shows that the pullback measure $\mu_t\circ\iota_*$ is independent of the choice of the embedding $\iota$, which concludes the proof.
\end{proof}

\begin{remark}\label{remark:marginal}
 In the proof of the previous lemma, we only used the connectivity of $G'$ for the existence of the symmetry $\Sigma$. Thus, when $G'$ is defined as a not necessarily connected (but finite) subgraph of the $d$-regular infinite tree $G$ with an embedding which is fixed up to the symmetries of the tree, the measure $\mu_t(G')$ is still well defined.
\end{remark}

To be able to formulate the evolution of the distributions over time, we need to have some kind of independence of the faraway parts of the graph. The next lemma clarifies the conditions for the independence.

\begin{lemma}
\label{lemma:condindep}
 Let $w_1$ and $w_2$ be two adjacent vertices of the $d$-regular infinite tree. If we remove the edge $w_1w_2$ the graph falls into two connected components $S_1$ and $S_2$. Let $A_1\subset S_1$ and $A_2\subset S_2$ be any finite set of vertices. Then for any step $t$, the colorings $C_t(A_1)$ and $C_t(A_2)$ are independent conditioned on the event $(C_{t-1}(w_1)="U")\cap(C_{t-1}(w_2)="U"),$ meaning that if there is an edge with both of its endpoint uncolored then the two sides of the edge are independent.
\end{lemma}

\begin{proof}
 With $i\in\{1,2\}$ and $0\leq s\leq t$ integer, define the set $A_i^s$ recursively as follows: $A_i^t=A_i$ and $A_i^s=N[A_i^{s+1}]\setminus\{w_{2-i}\}$ if $s<t$, i.e. the closed neighborhood of $A^{s+1}_i$ apart from the vertex $w_{2-i}$.

We prove the following stronger statement: the colorings $C_s(A_1^s)$ and $C_s(A_2^s)$ are independent conditioned on the event $(C_{t-1}(w_1)="U")\cap(C_{t-1}(w_2)="U").$
    
 It is true for $s=0$: in the beginning every vertex is uncolored which means that $C_0$ is constant. So $C_0(A_1^0)$ and $C_0(A_2^0)$ are also constant random variables, thus they are independent even when conditioned on any event.
    
 Now suppose that the statement is true for some $s<t$. We want the conditional independence of $C_{s+1}(A_1^{s+1})$ and $C_{s+1}(A_2^{s+1})$. With $i\in\{1,2\}$, let $T_i^{s+1}:A_i^{s+1}\to\{"U","R","B"\}$ be an arbitrary coloring on $A_i^{s+1}$. During the following calculation we use the following short hand notation for the event we condition on: $B_t=(C_{t-1}(w_1)="U")\cap(C_{t-1}(w_2)="U").$ Note that $B_t$ also implies $B_1,\dots,B_{t-1}.$ Now we can prove the conditional independence.
\begin{equation*}
\begin{split}
 \bbP\left(\bigcap_{i\in\{1,2\}}\{C_{s+1}(A_i^{s+1})=T_i^{s+1}\}\middle|B_t\right)&=\sum_{\substack{R^s_1\in\mathcal T^c_{A_1^s},\\R^s_2\in\mathcal T^c_{A_2^s}}}\bbP\left(\bigcap_{i\in\{1,2\}}\{C_{s+1}(A_i^{s+1})=T_i^{s+1},C_s(A_i^s)=R_i^s\}\middle| B_t\right)\\
 &=\sum_{\substack{R^s_1\in\mathcal T^c_{A_1^s},\\R^s_2\in\mathcal T^c_{A_2^s}}}\prod_{i\in\{1,2\}} \bbP\left(C_{s+1}(A_i^{s+1})=T_i^{s+1},C_s(A_i^s)=R_i^s\middle| B_t\right)\\
 &=\prod_{i\in\{1,2\}} \bbP\left(C_{s+1}(A_i^{s+1})=T_i^{s+1}\middle| B_t\right).\\
\end{split}
\end{equation*}

The first and the last equality is true because the sums run through disjoint events. For the equation in the middle, we need to understand what happens to a fixed vertex at one step. Based on the colors of its neighbors it either stays the same or using some randomness it changes color. The random choice of the vertices are independent of each other. So the behavior of one vertex at one step only depends on the state of its one-neighborhood and some randomness of its own. By fixing the coloring after the step $s$ as $C_s(A_i^s)=R_i^s$, we only need to check if the one-neighbors overlap or not. Apart from $w_1$ and $w_2$ any vertex $v_i\in A_i$ has $N[v_i]\subset S_i$. But $C_{s}(w_1)=C_{s}(w_2)="U"$ is forced by the condition $B_t$. So the colorings $C_{s+1}(A_1^{s+1})$ and $C_{s+1}(A_2^{s+1})$ are independent under the condition $B_t$ when we know the colorings $C_s(A_i^s)$.
\end{proof}

At the end of the first (greedy) phase of the algorithm, we need the distribution of the coloring of the one-neighborhood of the root of the $d$-regular infinite tree to determine whether it would be better to flip the color of the root or not. Due to Lemma \ref{lemma:condindep} we do not need to follow the coloring of very large neighborhoods. To calculate this distribution, it is enough to follow a marginal of the coloring of the two-neighborhood of the root: when a neighbor $w$ of the root gets colored, we forget the information (red, blue or uncolored) on the neighbors of $w$ -- of course, beside the root. So for uncolored neighbors of the root we take note of the colors of its outer neighbors and for colored first neighbors we just note their colors. Moreover, because of the law of the coloring of the infinite regular tree is invariant under the isomorphisms of the graph, it is indeed enough to control the second neighborhood of the root.

After these preliminaries, we can calculate the transition probabilities. To be able to use our previous notation we denote the $2$-depth $d$-regular tree $T_{d,2}$ -- the index $d$ is omitted when clear from context. Let $A\in\mathcal T^{pc}_{T_2}$ be a partial colored $2$-depth $d$-regular tree. We want to calculate $\mu_t(A)$: the probability of $A$ after having $t$ coloring rounds. Because of our previous observations, we are only interested in the partial colorings where the root and its neighbors are either red, blue, or uncolored: they cannot be unspecified and a second neighbor of the root is unspecified if and only if the first neighbor connecting it to the root is not uncolored: either red or blue.

Our algorithm for coloring the vertices of the infinite $d$-regular tree defines an inhomogeneous Markov process with a distribution $\mu_t$ at time $t$.  If we manage to give the transition probabilities of the stochastic process at each step, then we can recursively calculate the law $\mu_t$ at time $t.$ We focus on the marginals described before.
\subsection{Transition probabilities}\label{transition}
For determining the transition probabilities, at each step $t$ we need a global quantity the dominant vertex type $D_t=\{r_t^d,b_t^d\}$ which we define recursively using the following formula:
$$D_t=\max\left\{D\text{ type}\mid \mu_{t-1}(D)\geq\eps\right\},$$ where the maximum is taken with respect to the priority ordering introduced in Subsection \ref{subsec:greedy} and $\mu_{t-1}(D)$ is the total measure of colorings where the root is of type $D$ after step $t-1.$
When the set $\left\{D\text{ type}\mid \mu_{t-1}(D)\geq\eps\right\}$ is empty and $D_t$ cannot be defined, then the first (greedy) phase ends. We recall that the number of steps taken is denoted by $N$.
In most cases, it is easier to handle the types as unordered pairs, but when considering a particular vertex, it might be useful to know if it has more red or blue neighbors. In this "ordered" case, we write $D_t^r$ and $D_t^b$, which means that the given vertex is of type $D_t$ and has at least as many red neighbors as blue and at least as many blue neighbors as red, respectively. Note that when $r_t^d=b_t^d$ holds, then $D_t^r$ and $D_t^b$ are the same. For detecting this phenomenon, we define the \emph{multiplicity} of the dominant vertex type as
$$m(D_t)=\begin{cases}1
&\quad\text{if } D_t^r=D_t^b,\\
    2&\quad\text{if } D_t^r\neq D_t^b.\\
\end{cases}$$
We also must emphasize that $D_t$ is a deterministic and global quantity because it is calculated using all the colors of all the vertices of the infinite tree and it depends on the law $\mu_{t-1}$, not on the random variable $C_{t-1}$. If we run the algorithm in a local manner, the law $\mu_{t-1}$ is not available to the nodes during runtime, thus the dominant type cannot be determined. This is why we gave the sequence of dominant types as a parameter of our algorithm: to make it indeed local.

Let $A$ and $B$ two partial colored $2$-depth $d$-regular trees, i.e., elements of $\mathcal T^{pc}_{T_2}$. Then the transition probability at step $t$ from $A$ to $B$ (with dominant type $D_t$) is $\bbP(C_t\in B\mid C_{t-1}\in A)$ where we use the natural identification of partial colored $2$-depth $d$-regular trees with partial colored infinite $d$-regular trees by taking the color of the extra vertices to be unspecified as described in the proof of Lemma \ref{lemma:marginal}. The conditional probability splits due to the fact that the choices at the individual vertices are happening independently: either deterministically or with an independent random choice. Thus, we can write
\begin{equation}
 \label{eq:transition_probability}
 \bbP(
C_t\in B\mid C_{t-1}\in A)=\prod_{w\in V(T_2)}\bbP(C_t(w)\in B(w)\mid C_{t-1}\in A),
\end{equation}
where $C_t(w)\in B(w)$ is true if $B(w)$ is "unspecified" or $C_t(w)=B(w)$. As the symmetries of the $2$-depth $d$-regular tree suggest, we need to consider the terms on the RHS grouped by how far the vertex $w$ is from the root of the tree.

For calculating these transition probabilities, first, we need to handle two more simple cases. Let $q_t$ be the probability that the root of the infinite $d$-regular tree is colored red after the $t^{th}$ step, conditioned on that before the $t^{th}$ step it was uncolored and it was of the dominant type $D_t^r$ having at least as many red neighbors as blue. Then we have 
$$q_t=\frac{m(D_t)\cdot\eps/2}{\mu_{t-1}(D_t)}.$$
Indeed, at a step we want to color vertices of total measure $\eps.$ When the dominant type is symmetric, i.e. $m(D_t)=1,$ the formula reduces to $\frac{\eps/2}{\mu_{t-1}(D_t)},$ which is clearly the desired probability. When $m(D_t)=2,$ we have $\frac{\eps}{\mu_{t-1}(D_t)}=\frac{\eps/2}{\mu_{t-1}(D_t^r)}$ which gives again the correct probability. Because of the symmetry of the colors red and blue in our algorithm, $q_t$ also gives the probability that the root is blue after the $t^{th}$ step, conditioned that after the $t^{th}$ step it was uncolored and of the dominant type with at least as many blue neighbors as red. The sequence of these conditional probabilities is the sequence of threshold values in Subsection \ref{subsec:newalg}, the other parameter of the local algorithm.

For the second simple case, let us fix a neighbor $v_1$ of the root $v_0$ of the infinite $d$-regular tree. Then, let $\hat q_t$ be the probability that $v_0$, the root of the infinite $d$-regular tree, is red after the $t^{th}$ step, conditioned on that before the $t^{th}$ step neither the root $v_0$, nor its previously fixed neighbor $v_1$ were colored. Note that here we do not use the dominant type in the condition: in the condition, we do not know anything about the color of the other neighbors of $v_0$. Also, the condition is different from "both the root and any of its neighbors are uncolored": we fix the uncolored neighbor. We can calculate this conditional probability as
$$\hat q_t=\frac{(d-r_t^d-b_t^d)/d\cdot\eps/2}{\mu_{t-1}(\{v_0\mapsto\text{"U"}\}\cap\{v_1\mapsto\text{"U"}\})}.$$
As before, $r_t^d$ and $b_t^d$ is the number of a vertex of the dominant type, respectively. $\{v_0\mapsto\text{"U"}\}$ is the set of partial colorings of the infinite $d$-regular tree for which the root $v_0$ is uncolored. The other set is defined analogously. One can check the formula by writing out the definition of conditional probability. The denominator is just the probability of the condition, and in the numerator, we have the probability that the root is getting colored red (this gives $\eps/2$) at step $t$ when its prefixed neighbor $v_1$ was uncolored after step $t-1$ (this gives the normalizing factor). And again, by the symmetry of the colors red and blue, $\hat q_t$ is also the probability that the root is blue after the $t^{th}$ step, given the same conditions.

Once we obtain the measures $\mu_t$ of all labelled $2$-depth $d$-regular trees, we can determine the measure of the red-blue cut in the $d$-regular tree, taking into account the second phase of the algorithm (the recoloring) as well, via the observation below.

\begin{observation}
The density of the red-blue cut after the first (greedy) phase of the algorithm on the $d$-regular tree is $$\mu(\{v_0\mapsto "R"\}\cap\{v_1\mapsto "B"\}),$$ where $\mu$ is the law of the coloring of the infinite tree after finishing the first phase and $v_0$ is the root of the infinite tree and $v_1$ is one of its fixed neighbors (analogously to the definition of $\hat q_t$).
The measure of the miscolored vertices (c.f. the second phase) is $$2\cdot\mu(\{v_0\mapsto "R"\}\cap\{N(v_0)\text{ has more vertices of color blue than red}\}).$$
\end{observation}

Now, we give the formulas for the vertex-wise conditional probabilities from Equation \eqref{eq:transition_probability} as follows. We have two cases depending on how far the vertex $w$ is from the root of the tree.

\begin{enumerate}
\item case: $w$ is the root or a neighbor of the root.

\begin{tabular}{|c|c|c|c|}
 \hline
 conditions on $A$ \textbackslash value of $B(w)$& "R" & "B" & "U"\\
 \hline
 $A(w)=$"R" & $1$ & $0$ & $0$ \\
 $A(w)=$"B" & $0$ & $1$ & $0$ \\
 $A(w)=$"U" and $w$ is not of $D_t$ & $0$ & $0$ & $1$ \\
 $A(w)=$"U" and $w$ is of $D_t^r$ & $q_t$ & $(2-m(D_t))\cdot q_t$ & $1-(3-m(D_t))\cdot q_t$ \\
 $A(w)=$"U" and $w$ is of $D_t^b$ & $(2-m(D_t))\cdot q_t$ & $q_t$ & $1-(3-m(D_t))\cdot q_t$ \\
 \hline
\end{tabular}

Note that when $r_t^d=b_t^d$, then the last two rows are the same.

\item case: $w\neq v_0$ is a neighbor of a neighbor $v$ of the root.

\begin{tabular}{|c|c|c|c|}
 \hline
 conditions on $A$ \textbackslash value of $B(w)$& "R" & "B" & "U"\\
 \hline
 $A(w)=$"R" & $1$ & $0$ & $0$ \\
 $A(w)=$"B" & $0$ & $1$ & $0$ \\
 $A(w)=$"U" and $A(v)=$"U" & $\hat q_t$ & $\hat q_t$ & $1-2\hat q_t$\\

 \hline
\end{tabular}

Indeed, after a neighbor of the root is colored, we do not continue to follow the color of its other neighbors, the case when $A(v)$ is not "U" is not one to consider.
\end{enumerate}
Combining these and \eqref{eq:transition_probability} together, we have the transition probabilities $\bbP(C_t\in B\mid C_{t-1}\in A),$ so we can recursively calculate the laws $\mu_t.$ 

\subsection{Numerical results}
\begin{table}[h!!]
	\centering
	\begin{tabular} {|c||c|c|c|}
  \hline
  $1/\eps$ & \thead{cut size without\\improvement} & \thead{miscolored vertices} & \thead{improved\\cut size} \\
  
  \hline
  $10^3$ & $0.501778$ & $0.0199445$ & $0.501257$ \\
  
  \hline
  $10^4$ & $0.503125$ & $0.0190561$ & $0.497012$ \\
  
  \hline
  $5\cdot 10^4$ & $0.502832$ & $0.0187139$ & $0.496488$ \\
  
  \hline   \hline
  $5\cdot 10^4(^*)$ & $0.502803$ & $0.018679$ & $0.496488$ \\
  \hline
  $2.5\cdot 10^6(^*)$ & $0.502743$ & $0.0186205$ & $0.496392$ \\
  \hline
	\end{tabular}
	\caption{Our numerical results for the $5$-regular case. The last two rows are calculated using simplified formulas, see Remark \ref{simplified}.}
    \label{tab:results}
\end{table}
We made a computer program to calculate $\mu_t$ recursively\footnote{The source is available here: \url{https://github.com/szl21/starCounting}}. Our numerical results are in Table \ref{tab:results} for the $5$-regular case. 

Accordingly, we have that for $1/\eps=5\cdot 10^4$ the improved cut size is at most $0.496488$ which concludes the proof of Theorem \ref{bisect}, provided that we show that the cut-size provided by the process on random $d$-regular $n$-vertex graphs approximates the respective parameter given by the process on the infinite $d$-regular tree with high probability. This argument will be carried out in the next subsection. Notice that this result gives back the numerical bound of Díaz, Serna and Wormald 
 \cite{diaz2007bounds} for the cut size after the first phase. With the second phase we get better numerical values.

\begin{remark}\label{simplified}
    The last two rows of Table \ref{tab:results} are special in the following sense. As the form of Equation \eqref{eq:transition_probability} suggests, the transition probabilities are polynomial in the values $q_t$ and $\hat q_t,$ which are in turn a rational function of the law $\mu_{t-1},$ thus the calculation is far from just applying some linear operator multiple times. At each step we choose every vertex independently if they get colored or not. When the graph is huge, then the probability of two vertices get colored at a particular step inside a particular $2$-neighborhood is small. However, if we omit the higher order terms based on this intuition, we could not prove that the accumulation of the kind of error is controllable. The last two rows are calculated using the simplified formulas. As one can see in the table, numerical evidence suggests that the total error should remain small. If this is indeed true, then using the same amount of computational power we would be able to give better bounds. In fact, the time to compute the third and last rows of Table \ref{tab:results} were approximately the same.
\end{remark}

\subsection{Error estimation}
Now we discuss how our analysis on the infinite tree could differ from the process on a realization of a random graph. The problem is that near a short cycle the graph is not locally treelike.

For a given vertex $v$ and a time $t$, the color of $v$ after the $t^{th}$ step depends on the past of the $(t-1)$-neighborhood of  $v$. Indeed, the first step only depends on the first random label on $v$, all of its neighbors start as uncolored. Later, at step $t$, we know that after step $t-1$ the color of $v$ and all of its neighbors depended on their $t-2$-neighborhood and the $t^{th}$ step of $v$ depends on the current state of its one-neighborhood, which in turn depends on the past of its $t-1$-neighborhood. We take at most $1/\eps$ steps, so we need that the $(1/\eps+1)$-neighborhood of $v$ does not contain a cycle. Using a similar argument, for following the $2$-neighborhood of the vertex $v$, we need that its $(1/\eps+3)$ neighborhood does not contain a cycle. We denote this radius by $R=1/\eps+3.$

We find the vertices which does not satisfy the previous condition the other way around. Let us take a cycle of length $k.$ It is contained in the $R$-neighborhood of at most $$k\cdot\left(d-2+(d-2)\cdot(d-1)+\dots+(d-2)\cdot(d-1)^{R-1-\lfloor\frac{k+1}{2}\rfloor}\right)<k\cdot (d-1)^{R-\lfloor\frac{k+1}{2}\rfloor}$$ vertices. For the number of cycles of length $k,$ we use the following well-known statement.
\begin{lemma}[Bollobás \cite{bollobas1980probabilistic}, Wormald \cite{wormald1999models}]
\label{lemma:BW_poisson}
For every $k\ge 1$ and $d \ge 3$, the number of cycles of length $k$ in
a random $d$-regular graph converges in distribution to a Poisson random variable with mean equal to $\lambda_k=\frac{(d-1)^k}{2k}$.
\end{lemma}

Let us denote the number of cycles of length $k$ by $X_k.$ Then the number of vertices which has a cycle of length $k$ in their $R$-neighborhood is at most $X_k\cdot k\cdot (d-1)^{R-\lfloor\frac{k+1}{2}\rfloor},$ hence the number of nodes containing any cycle in their $R$-neighborhood is at most $\sum_{k=3}^{2R}X_k\cdot k\cdot (d-1)^{R-\lfloor\frac{k+1}{2}\rfloor}.$ If we fix the number $N$ of the vertices and pick a random graph on $N$ vertices then if we pick a random vertex of this now fixed graph, the probability of taking a vertex which has a cycle in its $R$-neighborhood is at most $\frac1N\cdot \sum_{k=3}^{2R}X_k\cdot k\cdot (d-1)^{R-\lfloor\frac{k+1}{2}\rfloor}.$ By Lemma \ref{lemma:BW_poisson} the variable $X_k$ converges in law when we tend with $N$ to infinity so if we take a sequence of random graphs $(G_l)_{l=1}^\infty$ on increasing number of vertices and $p_l$ is the probability that the $R$-neighborhood of a random vertex of $G_l$ contains a cycle, then $\lim\limits_{l\to\infty}p_l=0$ holds almost surely. Probability $p_l$ also gives an upper bound on the $L^\infty$-distance of the empirical and the calculated distribution of the colored two-neighborhoods. This, in turn, implies that our calculation gives an asymptotically almost surely correct estimation of the empirical distribution of a realization of the algorithm. Note that we did not use anything about the behavior of the process on the vertices close to small cycles.

\section{Applications: maximum cut and internal partitions}\label{application}

Here we discuss some applications of our main result.

\subsection{Maximum cut in d-regular graphs}

 As an application, we show that our algorithm also gives a lower bound on the size of the maximal bisection in regular graphs. 
 
 \begin{proof}[Proof of Theorem \ref{maxcut}]
     
 Let us call our stochastic algorithm $A,$ which constructs a small bisection of a $d$-regular graph $G.$ If the cut in the bisection defined by $A$ includes $\lambda$ fraction of the edges in expectation, then there is another stochastic algorithm $B$ which gives a large bisection for which the cut includes $1-\lambda$ fraction of the edges in expectation. Moreover, the two algorithms are coupled in the sense that each edge of the graph $G$ is included in either the cut generated by algorithm $A$ or the cut generated by algorithm $B$ for a given realization.

 First, we need a definition. We say that a local bisection algorithm is \emph{symmetric} to the color classes if the following holds. There exists a measure-preserving involution $f$ of the space of seeds so that if we apply $f$ on every seed, then the output at the root changes to the other class. Algorithm $A$ is symmetric in this sense: when at step $t,$ the dominant type is symmetric, then $f$ swaps the intervals $(0,q_t]$ and $(q_t,2q_t]$; otherwise, it is the identity.

 The structure of algorithm $B$ is the same as algorithm $A,$ with the following modifications:
 \begin{itemize}
  \item instead of $\textbf s,$ we use $f(\textbf s)$ as seeds.
	\item each vertex is colorblind: it thinks that its neighbors are of the opposite color as they are really (the opposite of red is blue, and the opposite of uncolored is uncolored)
 \end{itemize}
 
 During the first phase, whenever an edge is getting monochromatic according to algorithm $A,$ it gets bicolored according to algorithm $B;$ and the other way around. In the second phase, exactly the same vertices are getting recolored, but they are increasing the size of the bisection instead of decreasing it.
 \end{proof}

\medskip
Note that this argument should not be considered an intuitive argument supporting the
Zdeborov\'a-Boettcher Conjecture \cite{zdeborova2010conjecture}.
Moreover, the first author conjectures the opposite for the following intuitive reason.
The minimum bisection problem could be attacked in the following way.
We first try to separate the vertices into two classes, class Red is slightly bigger than class Blue.
In the size ratio 1:0, the cut size is 0, so it suggests that a slight imbalance helps to construct a smaller cut than the balanced cut.
Now identify the finite sets of vertices in class Red which have exactly the same number of total degree to the two classes.
We expect a positive fraction of the vertices to fall into this category.
If we move some of these finite sets to class Blue, then we can make the construction balanced.

This intuition suggests that the minimum bisection has an asymmetrical structure.
In contrast, the first author expects that the maximum cut has a symmetric structure. 
He conjectures that the sum of the expected sizes of the maximum cut and the minimum bisection is less than the total number of edges.

\subsection{Internal partitions of 5-regular graphs}
We recall that an {\em internal} or {\em friendly} partition of a graph is a partition of the vertices into two nonempty sets so that every vertex has at least as many neighbors in its own class as in the other one. The problem of finding or showing the existence of internal partitions in graphs has a long history. The same concept was introduced by Gerber and Kobler \cite{gerber2004classes} under the name of {\em satisfactory partitions}, while Kristiansen, Hedetniemi, and Hedetniemi \cite{kristiansen2002introduction} considered a related problem on {\em graph alliances}. A survey of Bazgan, Tuza, and Vanderpooten \cite{bazgan2010satisfactory} describes early results on the area and discusses the complexity of the problem as well as how to find such partitions. Let us denote by $d_G(v)$ the degree of vertex $v$ in graph $G$. For a set $U\subset V(G)$, $d_U(v)$ denotes the number of neighbors of $v$ in $U$.\\
Stiebitz \cite{stiebitz1996decomposing} proved that for every pair of functions
$a, b : V \rightarrow \mathbb{N}^+$ such that $ d_G(v) \geq a(v) +b(v) + 1 $ holds for all $v\in V$, there exists a partition of the vertex set 
$V(G) = A\cup B$, such that  $d_A(v) \geq  a(v) \  \forall v \in A$ and  $d_B(v) \geq b(v) \ \forall v \in B$ is satisfied.
This confirms a conjecture of Thomassen \cite{thomassen1983graph} in a strong form. 
An internal or friendly partition of a graph is a partition of the vertex set into two nonempty sets so that every vertex has at least as many neighbors in its own class as in the other one. It has been shown that apart from finitely many counterexamples, every \(3, 4\) or \(6\)-regular graph has an internal partition, due to Shafique, Dutton, and Ban, Linial. 

\begin{theorem}[Shafique-Dutton \cite{shafique2002satisfactory}, Ban-Linial \cite{ban2016internal}]\label{kisr} Let $d\in \{3,4,6\}$. Then, apart from finitely many counterexamples, all $d$-regular graphs have internal partitions. The list of counterexamples is as follows.
\begin{itemize}
  \item for $d=3$, $K_4$ and $K_{3,3}$ do not have an internal partition \cite{shafique2002satisfactory}.
  \item for $d=4$, $K_5$  does not have an internal partition \cite{shafique2002satisfactory}.
  \item for $d=6$, every graph on at least $12$ vertices has an internal partition, thus counterexamples have at most $11$ vertices (and this bound is tight) \cite{ban2016internal}.
\end{itemize}
\end{theorem}
The general conjecture states that this holds for
the family of \(d\)-regular graphs as well, for all fixed \(d\).
For even valency, Linial and Louis proved \cite{linial2020asymptotically} that almost all $d$ regular graphs admit an internal partition if $d$ is even. Note that for even valency, $d=d(v)=d(v)/2+d(v)/2$ while for odd valency, $d=d(v)=\lceil d(v)/2\rceil +\lceil d(v)/2\rceil -1$, which points out the difference in difficulty depending on the parity, in view of Stiebitz's result \cite{stiebitz1996decomposing}.

Note that in the dense setting, an even stronger statement holds true. Very recently, Minzer, Sah, and Sawhney showed that in the Erdős-Rényi random graph $G\sim \mathbb{G}(n,1/2)$, there is a friendly partition with high probability which is an equipartition such that each vertex has $c\sqrt{n}$ more neighbors in its own part than in the other part \cite{minzer2024perfectly}.

This improved the previous work of Ferber, Kwan, Narayanan, Sah, and Sawhney \cite{ferber2022friendly}.

\begin{proof}[Proof of Theorem \ref{internal}]
Following the observation of Bazgan, Tuza, and Vanderpooten \cite{bazgan2010satisfactory},  Bärnkopf, Nagy, and Paulovics \cite{barnkopf2024note} proved that if there exists a bisection of a \(5\)-regular simple
graph $G$ of size at most \(n/2 + 5\), then there exists an internal partition for \(G\). 
Due to our Theorem \ref{bisect}, we know that the bisection width of random \(5\)-regular graphs of size \(n\) is asymptotically almost surely below \(0.496488n\), which implies Theorem \ref{internal}.
\end{proof}

\begin{remark}
This approach does not seem to be applicable for larger odd values of $d$, as we can not expect to have the bisection width to be under $0.5$ for $d$-regular graphs for $d\geq 7$. Observe that the result of Díaz, Serna, and Wormald was almost strong enough to show this, but as they proved that the bisection width is at most $0.5028n$ a.a.s. for $5$-regular graphs.
\end{remark}

\section{Conclusion}
It is natural to discuss the limitations of the proposed algorithm in terms of its effectiveness, both computationally and performance-wise.

Let us highlight first that while applying the second, recoloring phase, actually we could swap much more vertices than what we obtain from Theorem \ref{biholeswap}.
Indeed, the improvement in the size of the vertex cut in the $5$-regular case via the application of Theorem \ref{biholeswap} counted as if the neighbors of\textit{ every miscolored red or blue vertex} of terminal type $(2,3)$ or $(3,2)$ in the opposite class are also miscolored  blue resp. red vertices of terminal type $(3,2)$ resp. $(2,3)$. Actually, 
the majority of them have \textit{zero such neighbors} with high probability, which means that these vertices could have been all added to the largest bihole in the associated bipartite graph. For comparison, we calculated with approximately the third of them using the general result of \cite{axenovich2021bipartite}.
The expected size of these could be determined as well using our factor of IID approach, however to do so, we should monitor the evolution of the second neighborhoods, which would require much more involved and much more transmission formulae.

This leads us to the second, and possibly more significant room for improvement. 
This is to improve the local algorithms themselves by introducing a much more refined rule for the local coloring, based on the color distribution of the second neighborhood $N[N[v]]$ or even larger neighborhoods. 
While in theory, we could derive further improvement with our approach in such a way, the amount of calculations would set some limitations in this direction.\\

\textbf{Acknowledgement} We are grateful to Marcin Bria\'nski for the fruitful conversation on the topic in the preliminary stages of the project. 

\bibliographystyle{abbrv}
\bibliography{sample}

\end{document}